\documentclass[12pt, twoside, leqno]{article}
\usepackage{amsmath,amsthm}
\usepackage{amssymb,latexsym}
\usepackage{enumerate}

\newtheorem{fed}{Definition}[section]
\newtheorem{theorem}[fed]{Theorem}
\newtheorem{proposition}[fed]{Proposition}
\newtheorem{corollary}[fed]{Corollary}
\newtheorem{definition}[fed]{Definition}
\newtheorem{lemma}[fed]{Lemma}
\newtheorem{remark}[fed]{Remark}

\newcommand{\HH}{\mathcal{H}}

\newcommand{\cS}{\mathcal{S}}

\begin{document}


\baselineskip=17pt


\title{Weighted projections into  closed subspaces}

\author{G. Corach\\
Departamento de Matem\'atica, FI-UBA and\\
Instituto Argentino de Matem\'atica\\
Buenos Aires, Argentina\\
E-mail: gcorach@fi.uba.ar
\and
G. Fongi\\
Instituto Argentino de Matem\'atica\\
Buenos Aires, Argentina\\
E-mail: gfongi@gmail.com
\and
A. Maestripieri\\
Departamento de Matem\'atica, FI-UBA and\\
Instituto Argentino de Matem\'atica\\
Buenos Aires, Argentina\\
E-mail:amaestri@fi.uba.ar }

\date{}

\maketitle


\renewcommand{\thefootnote}{}



\renewcommand{\thefootnote}{\arabic{footnote}}
\setcounter{footnote}{0}


\begin{abstract}
In this paper we  study $A$-projections, i.e. operators of a Hilbert space $\HH$ which act as projections when a seminorm is considered in $\HH$. $A$-projections were introduced by Mitra and Rao \cite{[MitRao74]} for finite dimensional spaces. We relate this concept to the  theory of compatibility between positive operators and closed subspaces of $\HH$. We also study the relationship between  weighted least squares problems and compatibility.
\end{abstract}


\section{Introduction}

In 1974, S.K. Mitra and C.R. Rao \cite{[MitRao74]} introduced the notion of projection into a subspace with respect to a seminorm. More precisely, given a positive (semidefinite) matrix $A\in \mathbb{C}^{n\times n}$ and   a subspace $\mathcal S$ of $\mathbb{C}^n$, a matrix $T\in \mathbb{C}^{n\times n}$ is called an \emph{$A$-projection into $\mathcal S$} if $R(T)\subseteq \mathcal S$ and
$$
\|y-Ty\|_A\leq \|y-s\|_A, \qquad \textrm{ for all }\; y\in \mathbb{C}^n, s\in \mathcal S,
$$
where $\|z\|_A:=\langle Az, z\rangle ^{1/2}=:\langle z, z\rangle_A ^{1/2}$. Notice that  an $A$-projection $T$ need not to be an idempotent, but $AT^2=AT$. This notion is related to very general least squares problems and Mitra and Rao have found several applications in statistics, in particular in linear models, see also \cite{[RaoMit73],[TakTiaYan07],[TiaTak09]}.

In 1994,  S. Hassi and K. Nordstr\"om \cite{[HasNor94]} started the study of projections onto closed subspaces in Hilbert spaces, which are orthogonal with respect to an indefinite seminorm. Their paper suggested the notion of compatibility, proposed by G. Corach, A. Maestripieri and D. Stojanoff \cite{[CorMaeSto01],[CorMaeSto02],[CorMaeSto02_splines]}. A closed subspace $\mathcal S$ of a Hilbert space $\HH$ is said to be \emph{compatible} with a positive (semidefinite bounded linear) operator $A$  on $\HH$ if there exists a (bounded linear) projection $Q$ acting on $\HH$ such that $\mathcal S$ is the image of $Q$ and $AQ=Q^*A$. This equality means that $Q$ is selfadjoint with respect to the semi-inner product defined by $A$. This notion  has several applications in generalized contractions \cite{[CasSuc08],[Suc06],[Suc09]}, Krein space operators \cite{[HasNor94],[MaMPe06]}, frame theory \cite{[AntCorRuiSto05]}, least squares problems \cite{[CorMae05]}, signal processing \cite{[Eld03],[Eldwer05]} and so on. It should be noticed that non compatible pairs exist only if $\HH$ has infinite dimension \cite[6.2]{[CorMaeSto02_splines]}.  Therefore, in order to study the relationship between Mitra-Rao's theory with the compatibility results, which is the main goal of this paper, it is necessary to extend that theory to the infinite dimensional case.

Sections 2 contains notations and preliminary results needed in the sequel, in particular the well-known Douglas factorization  theorem \cite{[Dou66],[FilWil71]}. Section 3 contains a short resum\'e of definitions and the main results of compatibility theory with no proof. In particular, if $(A, \mathcal S)$ is compatible then a description of the set
$$
\mathcal P(A, \mathcal S)=\{Q\in L(\HH): Q^2=Q, R(Q)=\mathcal S, AQ=Q^*A \}
$$
is presented. Section 4 is devoted to develop the theory of $A$-projections in the context of infinite dimensional Hilbert spaces. We only include proofs if they are not similar to those for finite dimensional spaces provided by Mitra and Rao \cite{[MitRao74],[RaoMit73]}. The set $\Pi(A, \mathcal S)=\{T\in L(\HH): T \textrm { is an } A\textrm{-projection into }\\ \mathcal S \}$ is described and the precise relationship between $\mathcal P(A, \mathcal S)$ and $\Pi(A, \mathcal S)$ is presented, in the main result of the section, together with some minimality properties.
Section 5 deals with least squares problems.  An operator $G\in L(\HH)$ is called an  \emph{$A$-inverse} of a closed range operator $B$ if  for each $y \in \HH$, $Gy$ is an $A$-LSS of $Bx=y$, i.e.
$$
\|BGy-y\|_{A} \leq \|Bx-y\|_{A}, \quad x \in \HH.
$$We show that the existence of an $A$-inverse of an operator $B$ is equivalent to the compatibility of the pair $(A, R(B))$. Moreover the set of all $A$-inverses of $B$ is described. The second part of this section deals with restricted $A$-inverses of a certain $B$: $G\in L(\HH)$ is called an $A$-inverse of $B$ restricted to  $\mathcal M$  if $R(G)\subseteq\mathcal M$ and
$$
\|BGy-y\|_A\leq \|Bx-y\|_A, \quad \forall x \in \mathcal M.
$$
This notion, also due to Rao and Mitra \cite{[RaoMit73]}, is completely described in terms of some compatibility conditions. In particular, there exists such a $G$ if and only if $(A, B(\mathcal M))$ is compatible. The final part deals with the least squares solution of an equation like
$$
Bx=y
$$
where the vectors $x$'s are measured with the seminorm $\|\,\|_{A_1}$ defined by $A_1\in L(\HH)^+$ and the vectors $y$'s are measured with $\|\,\|_{A_2}$, for another $A_2\in L(\HH)^+$. Again, the situation is completely described by using certain compatibility conditions. Analogous problems have been considered in \cite{[CorMae05]} and \cite{[GirMaeMPe10]}. It should also be mentioned that L. Eld\'en \cite{[Eld80]} was the first to study this problem in finite dimensions.

\section{Preliminaries}

Throughout, $\HH$ denotes a  separable complex Hilbert space, $L(\HH)$  the algebra of
linear bounded operators of $\HH$  and $L(\HH)^+$ the cone of positive operators.
 Also, $\mathcal Q$ denotes the subset of $L(\HH)$ of oblique projections, i.e., $\mathcal Q=\{Q\in L(\HH),\, Q^2=Q\}$ and $\mathcal P$ the set of orthogonal projections, i.e. $\mathcal P=\{P\in L(\HH): P^2=P=P^*\}$.

For every $A \in L(\HH)$, $R(A)$ denotes the range of $A$ and $N(A)$ its nullspace.
Given  $\mathcal{M}$ and $\mathcal{N}$ two closed subspaces of $\HH$, then  $\mathcal{M}\dot{+}\mathcal{N}$ denotes the direct sum  of $\mathcal M$ and $\mathcal N$,  $\mathcal{M}\oplus \mathcal{N}$  the orthogonal sum and $\mathcal M\ominus \mathcal N=\mathcal M \cap (\mathcal M\cap\mathcal N)^\perp$.
If $\mathcal{M}\dot{+}\mathcal{N}=\HH$,  denote by $P_{\mathcal{M}// \mathcal{N}}$  the  oblique projection
with range  $\mathcal{M}$ and nullspace $\mathcal{N}$; in particular, $P_{\mathcal{M}}=P_{\mathcal{M}// \mathcal{M}^\perp}$.

Given a closed range operator $A$,  $A^\dagger$  denotes  the \emph{Moore Penrose inverse} of $A$, i.e. $A^\dagger$ is the unique solution of the system
$$
AXA=A, \quad XAX=X, \quad (AX)^*=AX, \quad (XA)^*=XA.
$$
%

Given  a closed subspace $\mathcal S$ of $\HH$, then $P_{\mathcal{S}}$ induces a matrix decomposition  for each $A \in L(\HH)$ as
follows: if $P=P_\cS$ then  $A \in L(\HH)$ can be written as
$$
A=\left(
    \begin{array}{cc}
     a_{11} & a_{12}\\
      a_{21} & a_{22} \\
    \end{array}
  \right),
$$
where $ a_{11}=PA{P}{|_{\mathcal{S}}} \in L(\mathcal{S})$, $
a_{12}=PA{(I-P)}{|_{{\mathcal{S}}^\bot}} \in
L(\mathcal{S}^\bot,\mathcal{S})$, $
a_{21}=(I-P)A{P}{|_{\mathcal{S}}} \in
L(\mathcal{S},{\mathcal{S}}^\bot)$ and $
a_{22}=(I-P)A{(I-P)}{|_{{\mathcal{S}}^\bot}} \in
L({\mathcal{S}}^\bot)$.
If $A\in L(\HH)^+$, then
\begin{equation}\label{desc matricial de op positivo}
A=\left(
    \begin{array}{cc}
     a & b\\
      b^* & c \\
    \end{array}
  \right),
\end{equation}
 with $R(b)\subseteq R(a^{1/2})$, see \cite{[AndTra75]}. Throughout this work, we will use the matrix representation of $A$ given by \eqref{desc matricial de op positivo}.

\medskip

Given  $A \in L(\HH)^+$, consider the following semi-inner product on $\HH$:
$$
\langle x,y \rangle_A=\langle Ax,y \rangle, \qquad x,y \in \HH.
$$
The seminorm associated  is given by
$$
\|x\|_A=\langle x, x\rangle_A^{1/2}=\|A^{1/2}x \|, \qquad x\in \HH.
$$

 An operator $C\in L(\HH)$  is called  \emph{$A$-selfadjoint}  if $\langle Cx, y\rangle_A= \langle x, Cy\rangle_A$ for all $x, y\in\HH$, or
equivalently $AC=C^*A$. 

\medskip

The following result, due to R. G. Douglas, characterizes the operator range inclusion.

\begin{theorem} (Douglas).
 Consider Hilbert spaces $\HH, \mathcal K, \mathcal G$ and operators  $A\in L(\HH, \mathcal G), B\in L(\mathcal K, \mathcal G)$. The following conditions are equivalent:
\begin{enumerate}
 \item $R(B)\subseteq R(A)$,
 \item $BB^*\leq\lambda AA^*$, for some $\lambda>0$,
 \item the equation $AX=B$ has a solution in  $L(\mathcal K,\HH)$.

In this case, there exists a unique $D\in L(\mathcal K,\HH)$ such that $AD=B$ and $R(D)\subseteq \overline{R(A^*)}$;
moreover, $\|D\|^2=\inf \{\lambda >0:~ BB^*\le \lambda AA^*\}$ and $N(D)=N(B)$. This solution is called the reduced solution of $AX=B$.
\end{enumerate}
\end{theorem}

The reader is referred to \cite[Theorem 1]{[Dou66]} and \cite[Theorem 2.1]{[FilWil71]} for the proof of Douglas' theorem.

\bigskip

\section{Compatibility}

 Given $A\in L(\HH)^+$ and $\mathcal S$ a closed subspace of $\HH$, consider the following set
 $$
 \mathcal P(A, \mathcal S)=\{ Q\in \mathcal Q: R(Q)=\mathcal S, AQ=Q^*A\}.
 $$

 The pair $(A, \mathcal S)$ is called \emph{compatible} if the set  $\mathcal P(A, \mathcal S)$ is not empty, or equivalently,  if there exists a projection $Q\in \mathcal Q$ with range $\mathcal S$ such that $AQ=Q^*A$.

 The following proposition collects some results about compatibility that can be found in \cite{[CorMaeSto02],[CorMaeSto06]}.

\begin{proposition}\label{caracterizaciones par compatible}
Consider $A\in L(\HH)^+$ with matrix form given by equation (\ref{desc matricial de op positivo}) and $\mathcal S$ a closed subspace of $\HH$.
 \begin{enumerate}
  \item If  the pair $(A, \mathcal S)$ is compatible, then $\mathcal S+ N(A)$ is closed.
  \item If $A\in L(\HH)^+$ has closed range and $\mathcal S+ N(A)$ is closed, then $(A, \mathcal S)$ is compatible.
  \item The  pair $(A, \mathcal S)$ is compatible if and only if $\HH=\mathcal S + A(\mathcal S)^\perp$.
  \item  The  pair $(A, \mathcal S)$ is compatible if and only if $R(b)\subseteq R(a)$.
 \end{enumerate}
\end{proposition}

\smallskip

As a consequence of Douglas' theorem and item 4 of  the above proposition, we obtain the following characterization of the set $\mathcal P(A, \mathcal S)$, see \cite{[CorMaeSto01]} for details.

\begin{corollary}\label{car matricial PAS}
Consider $(A, \mathcal S)$ compatible, then
$$
\mathcal P(A, \mathcal S)=\{\left(
    \begin{array}{cc}
     1 & x\\
      0 & 0 \\
    \end{array}
  \right): x\in L(\mathcal S^\perp, \mathcal S) \,\;  and \,\; ax=b\}.
$$
\end{corollary}

If the pair $(A, \mathcal S)$ is compatible, there is a distinguished element $P_{A, \mathcal S}\in \mathcal P(A,\mathcal S)$, namely the unique projection onto $\mathcal S$ with kernel $A(\mathcal S)^\perp\ominus \mathcal N$, where $\mathcal N=A(\mathcal S)^\perp\cap\mathcal S=N(A)\cap \mathcal S.$ By \cite[Proposition 4.1]{[CorMaeSto02_splines]}, $P_{A,\mathcal S}=P_{A, \mathcal S\ominus \mathcal N}+P_{\mathcal N}$ and $P_{A, \mathcal S\ominus \mathcal N}=P_{\mathcal S\ominus \mathcal N//A(\mathcal S)^\perp}$. Then the matrix decomposition of $P_{A, \mathcal S}$ induced by $P_\mathcal S$ is  given by
$$
P_{A, \mathcal S}=\left(
    \begin{array}{cc}
     1 & d\\
      0 & 0 \\
    \end{array}
  \right),
$$
where $d\in L(\mathcal S^\perp, \mathcal S)$ is the reduced solution of $ax=b$.

It is easy to see that the pair $(A, \mathcal S)$ is compatible if and only if the pair $(A, \mathcal S\ominus \mathcal N)$ is compatible.

%
%
%
%
%


\medskip
%
%

\section{Weighted projections}

Along this work $A$ is a positive bounded operator, i.e. $A\in L(\HH)^+$ and $\mathcal S$ is a closed subspace of $\HH$.

The following definition is due to Mitra and Rao for operators acting on finite dimensional Hilbert spaces, see \cite{[MitRao74]}.

\begin{definition}\rm An operator $T\in L(\HH)$ is called an
\emph{$A$-projection into $\mathcal{S}$} if
 $R(T) \subseteq \mathcal{S}$  and
 \begin{equation}\label{definicion A-proy}
 \|y-Ty\|_A\leq \|y-s\|_A, \quad \textit{ for all } y \in \HH,\quad \textit{  for all } s \in  \mathcal{S}.
 \end{equation}
$T$ is called  an \emph{$A$-projection} if $T$ is an $A$-projection into $\overline{R(T)}$.
\end{definition}

An $A$-projection into $\mathcal S$ is also called an \emph{$A$-weighted least squares process}, see \cite{[CorGirMae09],[Sar52]}.

\begin{remark}\rm
It is not difficult to see that inequality (\ref{definicion A-proy}) alone does not  imply the boundedness of $T$. Indeed, if  $A$ has infinite dimensional nullspace it is enought to consider $T=T_1P_{N(A)}+P_{\overline{R(A)}}$, with  $T_1: N(A)\rightarrow N(A)$ unbounded.
Similarly, it can  be proved that the range of an  $A$-projection is not necessarily closed.
\end{remark}

\begin{definition}\rm The operator $T\in L(\HH)$ is an \emph{$A$-idempotent} if $AT^2=AT$.
\end{definition}

Observe that the definition of $A$-idempotent only depends  on $N(A)$ in the sense that if $A, B\in L(\HH)$ are such that $N(A)=N(B)$ then $T$ is $A$-idempotent if and only if $T$ is $B$-idempotent.

The next two propositions generalize some of the results in \cite{[MitRao74]}. The proofs  for infinite dimensional Hilbert spaces follow essentially the same steps.

\begin{proposition}\label{caracterizacion1 A-proyeccion}
Let $T\in L(\HH)$. The following statements are equivalent:
\begin{enumerate}
 \item  $T$ is an $A$-projection,
 \item $T^*AT=AT$,
 \item $AT=T^*A$ and $AT^2=AT$; or equivalently, $T$ is an $A$-selfadjoint and also  $A$-idempotent.
\end{enumerate}
\end{proposition}
\begin{proof} See \cite[Lemma 2.1 and Lemma 2.2]{[MitRao74]}.
%
%
\end{proof}
\medskip

\medskip

\begin{proposition}\label{caracterizacion1 Pi(A,S)}
Consider $T\in L(\HH)$ such that $R(T)\subseteq  \mathcal S$. The following conditions are equivalent:
\begin{enumerate}
 \item $T$ is an $A$-projection  into $\mathcal{S}$,
 \item  $AT=T^*A$ and  $ATP_{\mathcal S}=AP_{\mathcal S}$,
 \item $P_{\mathcal S}AT=P_{\mathcal S}A$.
\end{enumerate}
\end{proposition}
\begin{proof}
\textit{1} $ \rightarrow $ \textit{2}: Let $T$ be an $A$-projection  into $\mathcal{S}$. In particular $T$ is an $A$-projection. Then by Proposition \ref{caracterizacion1 A-proyeccion}, $AT=T^*A$. On the other hand, for each $y \in \HH$ it holds that $\|y-Ty\|_A=\|A^{1/2}y-A^{1/2}Ty\|\leq \|y-s\|_A$ for all $s\in \mathcal S$. In particular, given $x\in \HH$, then $\|A^{1/2}P_{\mathcal S}x-A^{1/2}TP_{\mathcal S}x\|\leq \|P_{\mathcal S}x-s\|_A$ for all $s\in \mathcal S$. Therefore, $A^{1/2}P_{\mathcal S}=A^{1/2}TP_{\mathcal S}$ , so that $AP_{\mathcal S}=ATP_{\mathcal S}$.

 \textit{2} $ \rightarrow $ \textit{3}: If $AT=T^*A$ and $ATP_{\mathcal S}=AP_{\mathcal S}$ then $P_{\mathcal S}A=P_{\mathcal S}T^*A=P_{\mathcal S}AT$, so that $P_{\mathcal S}A=P_{\mathcal S}AT$.

  \textit{3} $ \rightarrow $ \textit{1}: Since  $P_{\mathcal S}AT=P_{\mathcal S}A$, then $T^*AP_{\mathcal S}=AP_{\mathcal S}$ so that
  $T^*AT=AT=T^*A$ because $R(T)\subseteq \mathcal S$.
 Therefore, by Proposition \ref{caracterizacion1 A-proyeccion},  $T$ is an $A$-projection into $\overline{R(T)}$, then $\|y-Ty\|_A\leq \|y-Tx\|_A$ for all $x, y\in \HH$. It remains to prove that $\|y-Ty\|_A\leq \|y-P_{\mathcal S}x\|_A$ for all $x, y\in \HH$. Since $ATP_{\mathcal S}=AP_{\mathcal S}$, then $A^{1/2}TP_{\mathcal S}=A^{1/2}P_{\mathcal S}$,
 so that $\|y-Ty\|_A\leq \|y-TP_{\mathcal S}x\|_A=\|y-P_{\mathcal S}x\|_A$ for all $x, y \in \HH$.
\end{proof}

\smallskip

\begin{remark}\rm
By Proposition \ref{caracterizacion1 A-proyeccion} and  Proposition \ref{caracterizacion1 Pi(A,S)}, given $T\in L(\HH)$ such that $R(T)\subseteq \mathcal S$ it holds that $T$ is an $A$-projection into $\mathcal S$ if and only if $T$ is an $A$-projection and $ATP_{\mathcal S}=AP_{\mathcal S}.$
\end{remark}
\medskip

\begin{lemma}\label{T A-proy entonces I-T A-proy}
If $T\in L(\HH)$ is an $A$-idempotent ($A$-projection) then $I-T$ is an $A$-idempotent ($A$-projection).
\end{lemma}
\begin{proof}
If $T\in L(\HH)$ is an $A$-idempotent, then  $A(I-T)^2=A(I-2T+T^2)=A(I-T)$, i.e. $I-T$ is an $A$-idempotent.
If $T$ is an  $A$-projection then, by Proposition \ref{caracterizacion1 A-proyeccion}, $T$ is $A$-idempotent and $A$-selfadjoint. Hence $I-T$ is $A$-idempotent and  $A(I-T)=A-T^*A=(I-T)^*A$. Again, by Proposition \ref{caracterizacion1 A-proyeccion},  $I-T$ is an $A$-projection.
 \end{proof}

\medskip

The following  result characterizes  $A$-projections in terms of oblique projections.

\begin{lemma}\label{T A-proy sii P_AT proy A-autoadj}
Let $T\in L(\HH)$. Then $T$ is an $A$-projection if and only if $P_{\overline{R(A)}}T\in \mathcal Q$ and it is $A$-selfadjoint.
\end{lemma}
\begin{proof}
Let $T$ be an $A$-projection  and denote $P=P_{\overline{R(A)}}$. By Proposition \ref{caracterizacion1 A-proyeccion}, it holds that  $AT=T^*A$ and $AT=AT^2$, then $(PT)^2=A^\dagger ATPT = A^\dagger T^*APT
=A^\dagger T^*AT=A^\dagger AT=PT$ and $(PT)^*A=T^*PA=T^*A=AT=APT$. Conversely, if $PT\in \mathcal Q$ and it is $A$-selfadjoint then $AT=APT=(PT)^*A=T^*A$ so that $T$ is $A$-selfadjoint. Also, $AT^2=APT^2=(PT)^*AT=(PT)^*APT=A(PT)^2=APT=AT$ so that $T$ is $A$-idempotent. By Proposition \ref{caracterizacion1  A-proyeccion}, $T$ is an $A$-projection.
\end{proof}

\medskip

The next result shows that $A$-projections behave like orthogonal projections, under the seminorm induced by $A$, in the sense that for an $A$-idempotent, the condition of being $A$-selfadjoint is equivalent to being an $A$-contraction, or $A$-positive. For $A$-contractions see for example \cite{[CasSuc08]} and \cite{[Suc06]}.

\begin{proposition}\label{Lema de Krein general}
Consider $T\in L(\HH)$ such that $T$ is an $A$-idempotent. Then the following statements are equivalent:
\begin{enumerate}
  \item $T$ is $A$-selfadjoint (so that $T$ is an $A$-projection),
  \item  $R(I-T) \subseteq R(AT)^\perp$,
  \item $T$ is an $A$-contraction, i.e. $T^*AT\leq A$.
\end{enumerate}
\end{proposition}
\begin{proof}
$\textit{1} \rightarrow \textit{2}$:  Suppose that $AT=T^*A$. Consider $y \in R(I-T)$ and  $z\in \HH$  such that $y=z-Tz$. Then, for $x\in \HH$
$$
\langle ATx, y\rangle =\langle x, AT y\rangle =\langle x, AT(z-Tz)\rangle=0,
$$
because $AT^2=AT$. Therefore, $y \in R(AT)^\perp$.

$\textit{2} \rightarrow \textit{3}$: For $x, y \in \HH$,
$$
\langle ATx, y\rangle =\langle ATx, Ty+(I-T)y\rangle=\langle ATx,Ty\rangle=\langle T^*ATx,y\rangle
$$
because $R(I-T)\subseteq R(AT)^\perp$. Therefore, $AT=T^*AT=T^*A$ and $T$ is $A$-selfadjoint. Then $T$ is an $A$-projection. Also, by Lemma \ref{T A-proy entonces I-T A-proy}, $E=I-T$ is an $A$-projection so that $AE=AE^2=E^*AE \in L(\HH)^+$. Therefore, $A=A(T+E)=T^*AT+E^*AE\geq T^*AT$.

$\textit{3} \rightarrow \textit{1}$: Since $T^*AT\leq A$, by Douglas' theorem, the equation $A^{1/2}X=T^*A^{1/2}$ admits a solution.
Let $D$ be the reduced solution of $A^{1/2}X=T^*A^{1/2}$, i.e. $D$ satisfies $A^{1/2}D=T^*A^{1/2}$ and $R(D)\subseteq \overline{R(A)}$. Then, observe that
$$
A^{1/2}D^2=(T^*)^2A^{1/2}=T^*A^{1/2},
$$
because $T$ is an $A$-idempotent. Also, $R(D^2)\subseteq R(D)\subseteq \overline{R(A)}$. Therefore  $D^2$ is also a reduced solution of $A^{1/2}X=T^*A^{1/2}$, so that $D^2=D$ by the uniqueness of the reduced solution.
On the other hand, by Douglas' theorem,
$\|D\|^2
=\inf \{\lambda: T^*AT\leq \lambda A\}\leq 1,$ because $T^*AT \leq A$.
Since $D^2=D$ and $\|D\|\leq 1$, then automatically it holds that $D^*=D$, so that
$T^*A=A^{1/2}DA^{1/2}$ is selfadjoint, i.e. $T^*A=AT$.
\end{proof}

\medskip

\begin{corollary}
Let  $T\in L(\HH)$  be an $A$-idempotent. The following statements are equivalent:
 \begin{enumerate}
  \item $T$ is an $A$-projection,
  \item  $\|T\|_A=1$,
  \item $\langle Tx, x\rangle_A \geq 0, \quad \forall x\in \HH$, i.e. $T$ is $A$-positive.
  \end{enumerate}
\end{corollary}
\begin{proof}
$\textit{1.} \rightarrow \textit{2.}$: Since $A$ is an $A$-projection, by Proposition \ref{Lema de Krein general}, $T^*AT\leq A$. Then, for $x\in \HH$,
$$
\|Tx\|_A^2=\langle ATx, Tx\rangle=\langle T^*ATx, x\rangle \leq \langle Ax, x\rangle=\|x\|_A^2,
$$
so that $\|T\|_A\leq 1$. Also,
$$
\|T(Tx)\|_A=\|A^{1/2}T^2x\|=\|A^{1/2}Tx\|=\|Tx\|_A,
$$
because $T$ is $A$-idempotent. Therefore  $\|T\|_A=1$.

$\textit{2.} \rightarrow \textit{3.}$: Consider $T\in L(\HH)$ such that $AT=AT^2$ and $\|T\|_A=1$. Observe that
$$
\langle T^*ATx, x\rangle=\langle ATx, Tx\rangle=\|Tx\|_A^2 \leq \|x\|_A^2=\langle Ax, x\rangle,
$$
so that $T^*AT\leq A$, and by Proposition \ref{Lema de Krein general}, $T$ is an $A$-projection. By Proposition \ref{caracterizacion1 A-proyeccion}, it follows that $AT=T^*AT\in L(\HH)^+$.

 $\textit{3.} \rightarrow \textit{1.}$: Since $AT\in L(\HH)^+$, then $AT=T^*A$. Also $T$ is $A$-idempotent so that, by Proposition \ref{caracterizacion1 A-proyeccion}, $T$ is an $A$-projection.
\end{proof}

\bigskip

In the following paragraphs we study conditions for the existence of $A$-projections into $\mathcal S$ and we characterize the set of these projections.

Define
$$
\Pi(A, \mathcal S)= \{ T\in L(\HH): T \textrm{ is an } A\textrm{-projection into } \mathcal S\},
$$
and
$$
\Pi(A)= \{ T\in L(\HH): T \textrm{ is an } A\textrm{-projection} \}.
$$

By Proposition \ref{caracterizacion1 Pi(A,S)}, it follows that
\begin{equation}\label{caracterizacion2 Pi(A,S)}
\Pi(A, \mathcal S)=\{T\in L(\HH): R(T)\subseteq \mathcal S, \, AT=T^*A, \, ATP_{\mathcal S}=AP_{\mathcal S}\}
\end{equation}
and
$$
\Pi(A)=\{T\in L(\HH): AT=T^*A, \, AT^2=AT\}.
$$

In particular if $A=I$, then $\Pi(I, \mathcal S) =\{P_{\mathcal S} \}$ and $\Pi (I)  =\mathcal P.$

\bigskip

The next result gives a characterization of  $A$-projections into $\mathcal S$ in terms of the matrix decomposition induced by $P_{\mathcal S}$.
Recall that
$A=\left( \begin{matrix}
a & b \\
b^* & c
\end{matrix} \right)$ is the matrix representation of $A$, as in \eqref{desc matricial de op positivo}.

\begin{proposition}\label{caracterizacion matricial Pi(A,S)}
$
\Pi(A, \mathcal S)=\{ T\in L(\HH):
T=\left( \begin{matrix}
x & y \\
0 & 0
\end{matrix} \right), \, ax=a, \, ay=b\}.
$
\end{proposition}
\begin{proof}
By equation \eqref{caracterizacion2 Pi(A,S)}, $T\in \Pi(A, \mathcal S)$ if and only if
$R(T) \subseteq \mathcal{S}$, $T^*A=AT$ and  $ATP_{\mathcal S}=AP_{\mathcal S}$.
 Observe that $R(T)\subseteq \mathcal S$ if and only if  the matrix representation of $T$ induced by $P_{\mathcal S}$ is $T=\left( \begin{matrix}
x & y \\
0 & 0
\end{matrix} \right)$. In this case, $AT=T^*A$
if and only if
$
ax=x^*a, \; ay=x^*b \textrm{ and } b^*y=y^*b.
$
Also, $ATP_{\mathcal S}=AP_{\mathcal S}$ is equivalent to
$
ax=a \textrm{ and }  b^*x=b^*.
$
Then $T \in \Pi(A, \mathcal S)$ if and only if
\begin{equation}\label{condiciones matriciales1}
ax=x^*a, \; ay=x^*b, \;  b^*y=y^*b, \;  ax=a \; \textrm{ and }\;  b^*x=b^*.
\end{equation}
 It is not difficult to see that (\ref{condiciones matriciales1}) is equivalent to
$
 ax=a  \; \textrm{ and  }\; ay=b.
$
\end{proof}

\smallskip
\begin{corollary}\label{PAS en PiAS}
If the pair  $(A, \mathcal S)$ is compatible, then
$$
\mathcal{P}(A, \mathcal S) \subseteq \Pi(A, \mathcal S).
$$
\end{corollary}
\begin{proof}
It follows from  Corollary \ref{car matricial PAS} and Proposition \ref{caracterizacion matricial Pi(A,S)}.
 \end{proof}
\smallskip

Applying item \textit{3} of  Proposition \ref{caracterizacion1 Pi(A,S)}, we obtain the following

\begin{corollary}
$
\Pi(A, \mathcal S)=\{ T\in L(\HH): R(T) \subseteq \mathcal S \textit{ and } T \textit{ is a solution of}\\ \textit{ the equation } P_{\mathcal S}AX=P_{\mathcal S}A \}.
$
\end{corollary}
\medskip

The next result shows the relationship between the compatibility of the pair $(A, \mathcal S)$ and the existence of $A$-projections into $\mathcal S$.

\begin{proposition}\label{(A,S) cpt sii Pi(A,S) no vacio}
 The pair $(A, \mathcal S)$ is compatible if and only if there exists an $A$-projection into $\mathcal S$.
\end{proposition}
\begin{proof}
By Proposition \ref{caracterizacion matricial Pi(A,S)}, the set $\Pi(A, \mathcal S)$ is not empty if and only if the equation $ay=b$ admits a solution (observe that $ax=a$ always admits a solution). By  Douglas' theorem this is equivalent to the condition $R(b)\subseteq R(a)$, or equivalently by Proposition  \ref{caracterizaciones par compatible}, the pair $(A, \mathcal S)$ is compatible.
\end{proof}

\smallskip

\begin{remark}\rm \label{Pi(A,S) no vacio sii P(A,S)}
 By the above proposition, it holds that $\Pi(A, \mathcal S)\neq \emptyset$ if and only if
 $\mathcal P (A, \mathcal S) \neq \emptyset$.
\end{remark}
\bigskip

Recall that $\mathcal N=\mathcal S \cap A(\mathcal S)^\perp=\mathcal S \cap N(A)$.

\medskip

\begin{proposition}\label{T en Pi(A,S) sii P(S-N)T=P(A,S-N)}
Let $T\in L(\HH)$ with $R(T)\subseteq \mathcal S$. Then $T$ is an $A$-projection into $\mathcal S$ if and only if $(A,\mathcal S)$ is compatible and $P_{\mathcal S\ominus \mathcal N}T=P_{A, \mathcal S\ominus \mathcal N}$.
\end{proposition}
\begin{proof}
 Suppose $T$ is an $A$-projection into $\mathcal S$. Let $Q=P_{\mathcal S\ominus \mathcal N}T$. Then $R(Q)\subseteq \mathcal S$ and $AT=AQ$. Since $T$ is an $A$-projection, then $AQ=AT=T^*A=Q^*A$, so $Q$ is $A$-selfadjoint. Moreover $AQP_\mathcal S=ATP_{\mathcal S}=AP_{\mathcal S}$, then $Q$ is an $A$-projection into $\mathcal S$. Therefore $AQ^2=AQ$ so that $R(Q^2-Q) \subseteq N(A)\cap (\mathcal S\ominus \mathcal N)=\mathcal N\cap \mathcal N^\perp=\{0\}$, or equivalently $Q^2=Q$. Moreover, from $AQP_\mathcal S=AP_\mathcal S$ and $AQ=Q^*A$ it follows that $AQP_{\mathcal S\ominus \mathcal N}=Q^*AP_{\mathcal S\ominus \mathcal N}=Q^*AP_{\mathcal S}=AQP_{\mathcal S}=AP_{\mathcal S}=AP_{\mathcal S\ominus \mathcal N}$.  Therefore $A(QP_{\mathcal S\ominus \mathcal N}-P_{\mathcal S\ominus \mathcal N})=0$. Also $R(QP_{\mathcal S\ominus \mathcal N}-P_{\mathcal S\ominus \mathcal N})\subseteq \mathcal S\ominus\mathcal N$. Hence $R(QP_{\mathcal S\ominus \mathcal N}-P_{\mathcal S\ominus \mathcal N})\subseteq  N(A)\cap (\mathcal S\ominus \mathcal N)=\{0\}$, so that $QP_{\mathcal S\ominus \mathcal N}=P_{\mathcal S\ominus \mathcal N}$ and then $\mathcal S\ominus \mathcal N\subseteq R(Q)$. Therefore $R(Q)=\mathcal S\ominus \mathcal N$. Since $AQ=Q^*A$, $Q^2=Q$ and $R(Q)=\mathcal S\ominus \mathcal N$, it follows that $Q=P_{A, \mathcal S\ominus \mathcal N}$. Then $(A, \mathcal S\ominus\mathcal N)$ is compatible, so that $(A, \mathcal S)$ is compatible (see Section 3). Conversely, if $(A, \mathcal S)$ is compatible and $P_{\mathcal S\ominus \mathcal N}T=P_{A, \mathcal S\ominus \mathcal N}$, then $T=P_{A, \mathcal S\ominus \mathcal N}+P_{\mathcal N}T$, so that $AT=T^*A$ and $ATP_{\mathcal S}=AP_{A, \mathcal S\ominus \mathcal N}P_{\mathcal S}=AP_{\mathcal S\ominus \mathcal N}=AP_{\mathcal S}$. Then $T$ is an $A$-projection into $\mathcal S$.

\end{proof}

The following result shows that $\Pi(A, \mathcal S)$ is an affine manifold.

\begin{proposition}\label{Pi(A,S) variedad afin}
If the pair $(A, \mathcal S)$ is compatible, then
$$
\Pi(A, \mathcal S)=P_{A, \mathcal S} + L(\HH,\mathcal N).
$$
\end{proposition}
\begin{proof}
Let $T\in \Pi(A, \mathcal S)$, then by Proposition  \ref{T en Pi(A,S) sii P(S-N)T=P(A,S-N)}, it follows that $T 
=P_{A, \mathcal S\ominus \mathcal N}+P_{\mathcal N}T=P_{A, \mathcal S}+P_{\mathcal N}(T-I)\in P_{A, \mathcal S}+ L(\HH,\mathcal N)$, see Section 3.
Conversely, if $T=P_{A, \mathcal S}+W$ with $W\in L(\HH, \mathcal N) $,  then
 $P_{\mathcal S\ominus \mathcal N}T=P_{A, \mathcal S\ominus \mathcal N}$. By Proposition \ref{T en Pi(A,S) sii P(S-N)T=P(A,S-N)}, it holds that $T$ is an $A$-projection into $\mathcal S$.
\end{proof}

\smallskip
\begin{remark}\rm \label{R(AT)}
Given $T\in \Pi(A, \mathcal S)$,  observe that $AT=AP_{A, \mathcal S}$ since $\mathcal N \subseteq N(A)$. Hence $A(R(T))=R(AT)=A(\mathcal S)$.
\end{remark}

\medskip

A natural question is whether $\mathcal P(A, \mathcal S)$ equals $\Pi(A, \mathcal S)$. We prove now that this happens if and only if $\mathcal P(A,\mathcal S)$ and/or $\Pi(A, \mathcal S)$ has cardinal 1.

\begin{theorem}
Suppose that the pair  $(A, \mathcal S)$ is compatible. Then the following statements are equivalent:
 \begin{enumerate}
 \item  $\mathcal P(A, \mathcal S)= \Pi(A, \mathcal S), $
 \item $\mathcal N=\{0\}$.
 \item  $card (\Pi(A, \mathcal S))=1, $
 \item $ card (\mathcal P(A, \mathcal S))=1.$

 \end{enumerate}
\end{theorem}

\begin{proof}
\textit{1}$\rightarrow$\textit{2}: Suppose $\mathcal N \neq \{0\}$ and consider $T=P_{A, \mathcal S}+P_{\mathcal N}$. Then, by the previous theorem $T\in \Pi(A, \mathcal S)$. But it is not difficult to see that $T^2\neq  T$, so that $T\notin \mathcal P(A, \mathcal S)$.

\textit{2}$\rightarrow$\textit{3}: If $\mathcal N=\{0\}$, by the previous theorem  $\Pi(A, \mathcal S)=\{ P_{A, \mathcal S}\}$.

\textit{3}$\rightarrow$\textit{4}:  It follows by Corollary \ref{PAS en PiAS} and Remark \ref{Pi(A,S) no vacio sii P(A,S)}.

\textit{4}$\rightarrow$\textit{1}: If $card (\mathcal P(A, \mathcal S))=1$, by \cite[Theorem 3.5]{[CorMaeSto01]} it holds that $\mathcal N=\{0\}$ and $\mathcal P(A, \mathcal S)=\{P_{A, \mathcal S}\}$. Hence, by the previous result, $ \Pi(A, \mathcal S)= \{ P_{A, \mathcal S}\}=\mathcal P(A, \mathcal S)$.
\end{proof}

\smallskip

\begin{corollary}
If  $A$ is invertible, then
$\Pi(A, \mathcal S)=\mathcal P(A, \mathcal S)= \{P_{A,\mathcal S}\}$.
\end{corollary}

\begin{remark}\rm
Under the hypothesis of the above corollary, $P_{A, \mathcal S}$ can be compute as
$$
P_{A, \mathcal S}=A^{-1/2}P_{\overline{A^{1/2}(\mathcal S)}}A^{1/2},
$$
see \cite[Section 3]{[TiaTak09]} or, more generally \cite[Proposition 3.3]{[CorMaeSto05]}.
\end{remark}
\medskip

Some minimality properties of $P_{A, \mathcal S}$ respect to $\mathcal P(A, \mathcal S)$  are proved in \cite[Theorem 3.5]{[CorMaeSto01]} and \cite[Theorem 3.2]{[CorMaeSto02]}. The next result extends these properties to the set $\Pi(A, \mathcal S)$.

\begin{proposition}
Suppose that the pair $(A, \mathcal S)$ is compatible. Then
 \begin{enumerate}
 \item  $\| P_{A, \mathcal S} \|= \min \{ \|T\|: T  \in \Pi(A, \mathcal S)\}$.
 \item $\|(I-P_{A, \mathcal S})x\|\leq \|(I-T)x\|$, \quad for all $x \in \HH$ and every $T\in \Pi(A, \mathcal S)$.
 \end{enumerate}
\end{proposition}

\begin{proof}
\textit{1}.  Consider $T \in \Pi(A, \mathcal S)$. Then, by Proposition  \ref{T en Pi(A,S) sii P(S-N)T=P(A,S-N)}, $T=P_{A, \mathcal S\ominus \mathcal N} + W$ for some $W \in L(\HH,\mathcal N)$. Then
$$
\|Tx\|^2=\|P_{A, \mathcal S\ominus \mathcal N}x\|^2 + \|Wx\|^2 \geq \|P_{A, \mathcal S\ominus \mathcal N}x\|^2, \qquad \textrm{ for all } x \in \HH.
$$
Therefore, $\|T\|\geq
\|P_{A, \mathcal S\ominus \mathcal N}\|.$ Finally, observe that
 \begin{eqnarray*}
  \|P_{A,\mathcal S}\|^2 & = & \|P_{A, \mathcal S}(P_{A, \mathcal S})^*\|=\|P_{\mathcal N}+P_{A, \mathcal S\ominus \mathcal N}(P_{A, \mathcal S\ominus \mathcal N})^* \|\\
                    &= &\max \{\|P_{\mathcal N}\|, \|P_{A, \mathcal S\ominus \mathcal N} \| \}=\| P_{A, \mathcal S\ominus \mathcal N}\|^2,
 \end{eqnarray*}
 because $P_{A, \mathcal S}=P_{A, \mathcal S\ominus \mathcal N}+P_{\mathcal N}$, $P_{A, \mathcal S\ominus\mathcal N}P_{\mathcal N}=0=P_{\mathcal N}(P_{A, \mathcal S\ominus\mathcal N})^*$ and \\$P_{\mathcal N}P_{A, \mathcal S\ominus\mathcal N}=0=(P_{A, \mathcal S\ominus\mathcal N})^*P_{\mathcal N}$.

\textit{2}. Consider $T \in \Pi(A, \mathcal S)$. By Proposition \ref{Pi(A,S) variedad afin},  $T=P_{A, \mathcal S}+W$ for some  $W\in L(\HH, \mathcal N)$.  Observe that
$$
\|(I-T)x\|^2=\|(I-P_{A, \mathcal S})x\|^2+\|Wx\|^2\leq \|(I-P_{A, \mathcal S})x\|^2, \quad \textrm{ for all } x\in \HH,
$$
because $R(I-P_{A, \mathcal S})=N(P_{A, \mathcal S})=A(\mathcal S)^\perp\ominus \mathcal N \subseteq \mathcal N^\perp$
\end{proof}

\bigskip

As an application, we characterize the abstract splines in terms of weighted projections. The theory  of abstract splines is due to Atteia \cite{[Att65]}. The reader is referred to \cite{[CorMaeSto02_splines]} for some relationships between the notion of compatibility and abstract splines in Hilbert spaces.

Given $C\in L(\HH)$, $\mathcal S$ a closed subspace of $\HH$ and $x\in \HH$, an \emph{abstract spline} or a $(C,\mathcal S)$-\emph{spline interpolant} to $x$ is any element of the set
$$
sp(C, \mathcal S, x)=\{y\in x+\mathcal S: \|Cy \|=\min_{s \in \mathcal S} \|C(x+s) \| \}.
$$
If $A=C^*C\in L(\HH)^+$, observe that
$\|y\|_A=\|Cy\|$, for $y\in \HH$. Then
$$
sp(C, \mathcal S, x)=\{y \in x+\mathcal S: \|y \|_A= d_A(x, \mathcal S) \| \}.
$$
where $d_A(x, \mathcal S)=\inf_{s \in \mathcal S}\|x+s\|_A.$

\smallskip

The next proposition contains some results on splines, the proofs can be found in \cite{[CorMaeSto02_splines]}.

\begin{proposition}\label{propiedades de splines}
Given $C\in L(\HH)$, consider $  A=C^*C$. Then
\begin{enumerate}
 \item $sp(C, \mathcal S, x)=(x +\mathcal S)\cap A(\mathcal S)^\perp$, \; for $x\in \HH$.
 \item $sp(C, \mathcal S, x)$ is not empty for every $x\in \HH$ if and only if the pair $(A, \mathcal S)$ is compatible.
 \item If $(A, \mathcal S)$ is compatible and  $x \in \HH\setminus \mathcal S$, then  $sp(C, \mathcal S, x)=\{(I-Q)x: Q\in \mathcal P(A,\mathcal S) \}.$
\end{enumerate}
\end{proposition}

\smallskip

\begin{proposition}\label{caracterizacion splines con Pi(A,S)}
Consider $C\in L(\HH)$   and suppose that $(A, \mathcal S)$ is  compatible, where $A=C^*C$. For every nonzero  $x \in \HH$,  it holds that
$$
sp(C, \mathcal S, x)=\{ (I-T)x: T\in \Pi(A, \mathcal S)\}.
$$
\end{proposition}

\begin{proof}
Consider $y \in sp(C, \mathcal S, x)$. By item \textit{1.} of Proposition \ref{propiedades de splines} there exists $s\in \mathcal S $ such that $x=s+y$ and $y=x-s\in A(\mathcal S)^\perp$. We are looking for $T\in \Pi(A, \mathcal S)$ such that $(I-T)x=y$, or equivalently $Tx=s$. Note  that $y_1=(I-P_{A, \mathcal S})x\in N(P_{A, \mathcal S})\subseteq A(\mathcal S)^\perp$. Then
$$
y_1-y=(I-P_{A, \mathcal S})x-(x-s)=s-P_{A, \mathcal S}x \in \mathcal S\cap A(\mathcal S)^\perp =\mathcal N.
$$
Therefore, since $x\neq 0$, we can consider $W\in L(\HH, \mathcal N)$ sucht that $Wx=s-P_{A, \mathcal S}x$. By Proposition \ref{Pi(A,S) variedad afin} it follows that $T=P_{A, \mathcal S}+W \in \Pi(A, \mathcal S)$; moreover $Tx=s$.

Conversely, let $T\in \Pi(A, \mathcal S)$. Then $(I-T)x \in (x +\mathcal S)\cap R(I-T)$.
But, by  Proposition \ref{Lema de Krein general} and  Remark \ref{R(AT)}, $R(I-T)\subseteq R(AT)^\perp=A(\mathcal S)^\perp$. Therefore
$$
(I-T)x \in (x +\mathcal S)\cap A(\mathcal S)^\perp =sp(C, \mathcal S, x),
$$
by the above proposition.
\end{proof}

 Observe that, by  item \textit{1} of Proposition \ref{propiedades de splines},  $sp(C, \mathcal S,0)=\mathcal N$.

\section{Weighted  inverses }
Throughout this section,  $A\in L(\HH)^+$ and  $B\in L(\HH)$ is a closed range operator.

\begin{definition}\rm
 Given $y\in \HH$, $x_0\in \HH$ is an \emph{$A$-least squares solution} or an \emph{$A$-LSS} of $Bx=y$ if
\begin{equation} \label{A-LSS}
\|y-Bx_0\|_{A} \leq \|y-Bx\|_{A}, \quad x \in \HH.
\end{equation}
\end{definition}

\smallskip

\begin{remark}\rm \label{caracterizacion least squares solution}
Given $y\in \HH$,
$x_0$ satisfies \eqref{A-LSS} if and only if $\|A^{1/2}(y-Bx_0)\| \leq \|A^{1/2}(y-Bx)\|=\|A^{1/2}(y-Bx_0)+A^{1/2}B(x_0-x)\|$ for all $x\in \HH$, or equivalently $ \langle A^{1/2}y-A^{1/2}Bx_0, A^{1/2}Bz \rangle=0$ for all $z\in \HH$ (recall that given $a, b \in \HH$, it holds that $\|a\|\leq \|a+tb\|$ for all $t\in \mathbb{C} $ if and only if $\langle a, b \rangle =0$).
Then   $x_0$ is an $A$-LSS of $Bx=y$ if and only if  $x_0$ is a solution of
\begin{equation}\label{ecuacion normal}
B^*ABx=B^*Ay.
\end{equation}
Equation \eqref{ecuacion normal} is  the \emph{normal equation} associated to \eqref{A-LSS}.
\end{remark}

\medskip
The next two results generalize \cite[Proposition 4.4]{[CorMae05]} and  \cite[Lemma 4.6]{[CorMae05]}.

\begin{proposition}\label{caracterizacion A-LSS}
Suppose $(A, R(B))$ is compatible and consider $y\in \HH$, $y\neq 0$. Then $u\in \HH$ is an $A$-LSS of $Bx=y$ if and only if there exists $T\in \Pi(A,R(B))$ such that $Bu=Ty$.
\end{proposition}
\begin{proof}
Observe that $u\in \HH$ is an $A$-LSS of $Bx=y$ if and only if $\|Bu-y\|_A=\inf _{\sigma \in R(B)}\|\sigma +y\|_A$, or  $y-Bu \in sp(A^{1/2}, R(B), y)$; or equivalently,  by Proposition \ref{caracterizacion splines con Pi(A,S)}, $Bu=Ty$ for some  $T\in \Pi(A, R(B))$.
\end{proof}

\smallskip

\begin{corollary}
Let $(A, N(B))$ be compatible and consider $x_0\in N(B)^\perp$, $x_0\neq0$ and $u\in x_0+N(B)$. Then $\|u\|_{A}\leq \|x\|_{A}$ for all $x\in x_0+N(B)$ if and only if there exists $T\in \Pi(A, N(B))$ such that $u=(I-T)x_0$.
\end{corollary}

\begin{proof}
Since $x_0\in N(B)^\perp$ and $u\in x_0+N(B)$, then $u=x_0+ P_{N(B)}u$. Consider $T\in \Pi(A, N(B))$ such that $u=(I-T)x_0$, so that $P_{N(B)}u=-Tx_0$. Since $x_0\neq 0$, by the previous proposition it holds that $u$ is an $A$-LSS of $P_{N(B)}x=-x_0$, then  $\|P_{N(B)}u+x_0\|_{A}\leq \|P_{N(B)}x+x_0\|_{A}$ for all $x\in \HH$, or equivalently  $\|u\|_{A}\leq \|x\|_{A}$ for all $x\in x_0+N(B)$.
The converse follows by \cite[Lemma 4.6]{[CorMae05]}.
\end{proof}
\medskip

The following concept was introduced by Rao and Mitra for finite dimensional spaces, \cite{[RaoMit73]}.

\begin{definition}\rm
 An operator $G\in L(\HH)$ is called an  \emph{$A$-inverse} of $B$ if for each $y \in \HH$, $Gy$ is an $A$-LSS of $Bx=y$, i.e.
$$
\|y-BGy\|_{A} \leq \|y-Bx\|_{A}, \quad x \in \HH.
$$
\end{definition}

\smallskip

\begin{remark}\rm
If $G$ is an $A$-inverse of $B$ then $R(G)$ is not necessarily closed. In fact, if $A$ has infinite dimensional nullspace consider   $G_1\in L(N(A))$ such that $R(G_1)$ is not closed. It is easy to see that $G=G_1P_{N(A)}+P_{N(A)^\perp}$ is an $A$-inverse of $I$.
\end{remark}

\medskip

The following result gives a necessary and sufficient condition for an operator $B$ with closed range to admit an $A$-inverse.

\begin{proposition}\label{existe A-inversa de B sii (A,R(B)) cpt}
The operator $B$ admits an $A$-inverse  if and only if  \\$(A, R(B))$ is compatible.
\end{proposition}
\begin{proof}
Let $G\in L(\HH)$ be an $A$-inverse of $B$ and consider $T=BG$. Then $R(T)\subseteq R(B)$ and $\|y-Ty\|_{A}=\|y-BGy\|_{A} \leq \|y-Bx\|_{A}$ for all $ x\in \HH$, so that $T$ is an $A$-projection into $R(B)$. Then $(A, R(B))$ is compatible by  Proposition \ref{(A,S) cpt sii Pi(A,S) no vacio}.
 Conversely, if $(A, R(B))$ is compatible, using again Proposition \ref{(A,S) cpt sii Pi(A,S) no vacio}, let $T$ be an $A$-projection into $R(B)$.
Since $R(T)\subseteq R(B)$, by Douglas' theorem there exists $G\in L(\HH)$ such that  $T=BG$. Therefore,
$$
\|y-BGy\|_{A}=\|y-Ty\|_{A}\leq \|y-Bx\|_A,  \qquad  \textrm{for } x\in \HH,
$$
so that $G$ is an $A$-inverse of $B$.
 \end{proof}

\begin{remark}\rm\label{G es A-inv sii BG is A-proy}
It follows from the above proof that, if $G$ is an $A$-inverse of $B$ then $T=BG$ is an $A$-projection into $R(B)$. Conversely, given $T$ an $A$-projection into  $R(B)$, the solutions of $BX=T$ are $A$-inverses of $B$.
\end{remark}
\medskip

The next result gives  necessary and sufficient conditions for an operator $G\in L(\HH)$ to be an $A$-inverse of $B$.

\begin{proposition}\label{caract A-inversa}
Given $G\in L(\HH)$ then $G$ is an $A$-inverse of $B$ if and only if $B^*ABG=B^*A$.
\end{proposition}
\begin{proof}
 Let $G\in L(\HH)$ be an $A$-inverse of $B$. By Remark \ref{G es A-inv sii BG is A-proy}, it holds that $T=BG$ is an $A$-projection into $R(B)$. Hence, by Proposition \ref{caracterizacion1 Pi(A,S)}, it follows that $P_{R(B)}AT=P_{R(B)}A$  so that  $B^*ABG=B^*A$.
Conversely, consider $T=BG$, then $B^*AT=B^*A$, or equivalently, $P_{R(B)}AT=P_{R(B)}A$. Therefore, by Proposition \ref{caracterizacion1 Pi(A,S)}, $T=BG$ is an $A$-projection into $R(B)$. Finally, by Remark \ref{G es A-inv sii BG is A-proy}, $G$ is an $A$-inverse of $B$.
%
%
%
\end{proof}

\smallskip

\begin{corollary}
If $(A, R(B))$ is compatible, then the set of $A$-inverses of $B$ is
$$
(B^*AB)^\dagger B^*A+ L(\HH, N(B^*AB)).
$$

\end{corollary}
\medskip

\smallskip

\subsection{Restricted  weighted inverses}
\medskip

Throughout this paragraph,  $\mathcal M$ is a closed subspace of $\HH$ such that $B(\mathcal M)$ is closed, or equivalently, since $B$ has closed range,  $\mathcal M + N(B)$ is a closed subspace of $\HH$.

\begin{definition}\rm
An operator $G\in L(\HH)$ is called an \emph{$A$-inverse of $B$ restricted to $\mathcal M$} if $R(G)\subseteq \mathcal M$ and for each $y\in \HH$ it holds that
$$
\|y-BGy\|_{A}\leq \|y-Bx\|_{A}, \qquad  \forall x\in \mathcal M.
$$
\end{definition}

The concept of  $A$-inverses  restricted to $\mathcal M$ was introduced by Rao and Mitra \cite{[RaoMit73]} for finite dimensional spaces.

In what follows we  show that the existence of an $A$-inverse of $B$ restricted to $\mathcal M$ is equivalent to the compatibility of the pair $(A, B(\mathcal M))$.

\begin{lemma}\label{A-inversa M-restringida sii A-inversa de BPM}
An operator $G\in L(\HH)$ is an   $A$-inverse of $B$ restricted to $\mathcal M$ if and only if  $R(G)\subseteq \mathcal M$ and $G$ is an $A$-inverse of $BP_\mathcal M$.
\end{lemma}

\begin{proof} Straightforward.
\end{proof}

\begin{remark}\rm\label{caracterizacion A-inversa restringida}
By the previous lemma and Proposition \ref{caract A-inversa} applied to $BP_{\mathcal M}$, it holds that $G$ is an  $A$-inverse of $B$ restricted to $\mathcal M$ if and only if
 $R(G)\subseteq \mathcal M$ and $P_{\mathcal M}(B^*ABG-B^*A)=0 $.
\end{remark}

\begin{proposition}\label{una M-restricted A-inverse de B}
Suppose $(A, B(\mathcal M))$ is compatible and consider $T\in \Pi(A, B(\mathcal M))$. Then the reduced solution of
$$
B P_\mathcal MX=T,
$$
is an    $A$-inverse of $B$ restricted to $\mathcal M$.
\end{proposition}

\begin{proof}
Let $G_0$ be the reduced solution of $B P_\mathcal MX=T$, then $R(G_0)\subseteq N(BP_\mathcal M)^\perp$. By Remark \ref{G es A-inv sii BG is A-proy}, $G_0$ is an $A$-inverse of $BP_\mathcal M$. Since $N(BP_\mathcal M)=(\mathcal M \cap N(B))\oplus \mathcal M^\perp$, then $R(G_0)\subseteq
\mathcal M$. Therefore, by Lemma \ref{A-inversa M-restringida sii A-inversa de BPM}, $G_0$ is an    $A$-inverse of $B$ restricted to $\mathcal M$.
\end{proof}

\smallskip

\begin{corollary}
The operator $B$ admits an   $A$-inverse restricted to $\mathcal M$ if and only if $(A, B(\mathcal M))$ is compatible.
\end{corollary}
\begin{proof}
If $G$ is an   $A$-inverse of $B$ restricted to $\mathcal M$, then by Remark \ref{A-inversa M-restringida sii A-inversa de BPM}, $G$ is an $A$-inverse of $BP_\mathcal M$, so that $(A,B(\mathcal M))$ is compatible (see Proposition \ref{existe A-inversa de B sii (A,R(B)) cpt}). The converse follows by Proposition \ref{una M-restricted A-inverse de B}.
\end{proof}

\smallskip

\subsection{$A_1A_2$-inverses and weak weighted inverses}

Throughout this section, we consider $B\in L(\HH)$ a closed range operator and $A_1, A_2\in L(\HH)^+$.

\begin{definition}\rm
 An operator $G\in L(\HH)$ is called  an \emph{$A_1A_2$-inverse} of $B$ if $G$ is an $A_1$- inverse of $B$ and, for each $y \in \HH$,  $Gy$ has minimum $A_2$-seminorm  among the $A_1$-LSS of $Bx=y$.
\end{definition}

In \cite{[MitRao74],[RaoMit73]}, $A_1A_2$-inverses are called minimum seminorm semileast squares inverses in the context of finite dimensional spaces.

\smallskip

The next two results are proved in \cite{[MitRao74]}, for finite dimensional Hilbert spaces. The proofs, which follow the same ideas, are included  for the sake of completeness.

\begin{proposition}\label{caracterizacion A_1A_2-inversa}
Consider  $G\in L(\HH)$. Then $G$ is an $A_1A_2$-inverse of $B$ if and only if
\begin{enumerate}
 \item $B^*A_1BG=B^*A_1$,
 \item $R(A_2G) \subseteq N(A_1B)^\perp$.
\end{enumerate}
\end{proposition}

\begin{proof}
By Proposition \ref{caract A-inversa}, $G$ is an $A_1$-inverse of $B$ if and only if
$B^*A_1BG=B^*A_1$. Let $G$ be an $A_1$-inverse of $B$. It is remains to prove that  $Gy$ has minimum $A_2$-seminorm  among the $A_1$-LSS of $Bx=y$  for each $y \in \HH $ if and only if $R(A_2G) \subseteq \overline{R(B^*A_1^{1/2})}$.
Observe that given $y\in \HH$, by Remark \ref{caracterizacion least squares solution}, any $A_1$-LSS of $Bx=y$ can be written as
$$
x_0=\tilde{x}+P_{N(B^*A_1B)}z,
$$
 where $\tilde{x}=Gy$ is a solution of (\ref{ecuacion normal}) (i.e $B^*A_1B\tilde{x}=B^*A_1y$) and $z\in \HH$. Then
  $ \|Gy\|_{A_2} \leq \|Gy+P_{N(B^*A_1B)}z\|_{A_2}, $ for all $z\in \HH$,
if and only if  $ \|A_2^{1/2}Gy\|\leq \|A_2^{1/2}Gy+A_2^{1/2}P_{N(B^*A_1B)}z)\|, $ for all $z\in \HH$, or
  equivalently \\$\langle A_2Gy, P_{N(B^*A_1B)}z\rangle=0$ for all $z\in \HH$, or $P_{N(B^*A_1B)}A_2G=0$. Therefore, $G$ is an $A_1A_2$-inverse of $B$ if and only if $B^*A_1BG=B^*A_1$ and $R(A_2G)\subseteq N(B^*A_1B)^\perp=N(A_1B)^\perp$.
\end{proof}


\smallskip

\begin{proposition}\label{G A1A2inversa ent G weak A1A2 inversa}
If $G$ is an $A_1A_2$-inverse of $B$, then
\begin{enumerate}
 \item  $A_1BGB=A_1B, \quad A_1BG=(BG)^*A_1,$
 \item $A_2GBG=A_2G, \quad A_2GB=(GB)^*A_2.$
\end{enumerate}
\end{proposition}
\begin{proof}
If $G$ is an $A_1A_2$-inverse of $B$ then $G$ is an $A_1$-inverse of  $B$. Therefore, by Proposition \ref{caract A-inversa}, 
$B^*A_1BG=B^*A_1$ then $A_1BG=(BG)^*A_1BG \geq 0$ so that $BG$ is $A_1$-selfadjoint. Also, $A_1B=(BG)^*A_1B=A_1BGB$ and item \textit{1} holds.
To prove item \textit{2} observe that  $R(I-GB)\subseteq N(B^*A_1B)$ because $B^*A_1BGB=B^*A_1B$. Therefore for each $y\in \HH$,
by Remark \ref{caracterizacion least squares solution}, it follows that
$$
x=Gy+(I-GB)z, \quad z\in \HH
$$
  is a solution of the normal equation  \eqref{ecuacion normal} and then it is an $A_1$-LSS of $Bx=y$.
Since  $G$ is an $A_1A_2$-inverse of $B$, then
 $ \|Gy\|_{A_2} \leq \|Gy+(I-GB)z\|_{A_2}, $ for all $z\in \HH$,
or equivalently $ \|A_2^{1/2}Gy\|\leq \|A_2^{1/2}Gy+A_2^{1/2}(I-GB)z)\|, $ for all $z\in \HH$, then
   $\langle A_2Gy, (I-GB)z\rangle=0$ for all $z\in \HH$, or $G^*A_2(I-GB)=0$. Finally, in the same way as we did in item \textit{1},  $G^*A_2=G^*A_2GB$ implies item \textit{2} (actually, both conditions are equivalent).
\end{proof}

\medskip

\begin{corollary}
Suppose the pairs $(A_1, R(B))$ and $(A_2, N(A_1B))$ are compatible. Then
$$
G=(I-T_2)B^\dagger T_1
$$
 is an $A_1A_2$-inverse of $B$ for every $T_1\in \Pi(A_1, R(B))$ and for every $T_2\in \Pi(A_2,N(A_1B))$.
\end{corollary}
\begin{proof}
Consider $T_1\in \Pi(A_1, R(B))$, $T_2\in \Pi(A_2,N(A_1B))$ and $G=(I-T_2)B^\dagger T_1$. Then
$$
A_1BGB=A_1BB^\dagger T_1B-A_1BT_2B^\dagger T_1B= A_1T_1B=A_1B,
$$
because $R(T_1)\subseteq R(B)$, $R(T_2)\subseteq N(A_1B)$ and $A_1T_1P_{R(B)}=A_1P_{R(B)}$. Also, observe that
$$
\begin{array}{ccl}
(BG)^*A_1  & = &(BB^\dagger T_1-BT_2B^\dagger T_1)^*A_1=T_1^*A_1-(A_1BT_2B^\dagger T_1)^* \\
           & = & A_1 T_1=A_1BG.
 \end{array}
 $$
Finally, by Proposition \ref{Lema de Krein general} and Remark \ref{R(AT)}, it holds that $
R(A_2G)  \subseteq A_2R(I-T_2)\subseteq A_2[R(A_2T_2)^\perp]=A_2[(A_2N(A_1B))^\perp]=A_2[A_2^{-1}(N(A_1B)^\perp)]
         \subseteq N(A_1B)^\perp.$
Therefore,  by Proposition \ref{caracterizacion A_1A_2-inversa}, it follows that $G$ is an $A_1A_2$-inverse of $B$.
\end{proof}
\medskip

\begin{proposition}
The operator $B$ admits an $A_1A_2$-inverse if and only if the pairs $(A_1, R(B))$ and $(A_2, N(A_1B))$ are compatible.
\end{proposition}
\begin{proof}
Suppose  $B$ admits an $A_1A_2$-inverse. Then, by Proposition \ref{existe A-inversa de B sii (A,R(B)) cpt}, the pair $(A_1R(B))$ is compatible. The pair $(A_2, N(A_1B))$ turns out to be compatible by \cite[Proposition 3.9]{[GirMaeMPe10]}. The converse follows by the previous result.
\end{proof}

\medskip

\begin{definition}\rm
An operator $G\in L(\HH)$ is called a \emph{weak $A_1A_2$-inverse} of $B$ if satisfies
\begin{equation}\label{definicion weak A1A2-inversa}
\left\{\begin{array}{cc}
 A_1BGB=A_1B,\,  & A_1BG=(BG)^*A_1\\
 A_2GBG=A_2G, \, & A_2GB=(GB)^*A_2.
\end{array}\right.
\end{equation}

\end{definition}

\smallskip
 If $A_1=A_2=I$ and $G$ is a weak $A_1A_2$-inverse of $B$, then $G=B^\dagger$.

Observe that if $G\in L(\HH)$ is  an $A_1A_2$-inverse of $B$ then, by Proposition \ref{G A1A2inversa ent G weak A1A2 inversa}, $G$  is a  weak $A_1A_2$-inverse of $B$.

\begin{remark}\rm \label{caract weak A1A2-inversa}
Observe that \eqref{definicion weak A1A2-inversa} is equivalent to $B^*A_1BG=B^*A_1$ and $G^*A_2GB=G^*A_2$.
\end{remark}
\begin{lemma}
Consider $G \in L(\HH)$. Then $G$ is a weak $A_1A_2$-inverse of $B$ if and only if $G$ is an $A_1$-inverse of $B$ and $B$ is an $A_2$-inverse of $G$.
\end{lemma}

\begin{proof} Apply Remark \ref{caract weak A1A2-inversa} and Proposition \ref{caract A-inversa}.
\end{proof}

\smallskip

\smallskip

In  \cite{[CorMae05]}, the authors called  \emph{weighted generalized inverse} of $B$ to an operator  $C\in L(\HH)$ such that
  \begin{equation}\label{weighted generalized inverses}
  BCB=B, \quad CBC=C, \quad A_1BC=(BC)^*A_1, \quad A_2CB=(CB)^*A_2.
  \end{equation}

In \cite[Theorem 3.1]{[CorMae05]}, it is proved that  the pairs $(A_1, R(B))$ and $(A_2, N(B))$ are compatible if and only if $B$ admits a weighted generalized inverse. Observe that in this case, $C$ is a weak $A_1A_2$-inverse of $B$.

Also, it holds that $C$ is a weighted generalized inverse of $B$ if and only if $BC \in\mathcal P(A, R(B))$ and $I-CB\in \mathcal P(A, N(B))$.
In order to generalize this, we now consider the solutions of the system
\begin{equation}\label{weak w.g.i.}
\left\{\begin{array}{l}
 BG \in \Pi( A_1, R(B))\\
 I-GB\in \Pi(A_2, N(B)).
\end{array}\right.
\end{equation}

\smallskip

\begin{proposition}\label{caracterizacion [2]A1A2-PI}
Consider $G \in L(\HH)$, then $G\in L(\HH)$ is  a solution of (\ref{weak w.g.i.}) if and only if
\begin{equation}\label{weak w.g.i. 3 ecuaciones}
BGB=B, \quad A_1BG=(BG)^*A_1, \quad A_2GB=(GB)^*A_2.
\end{equation}
\end{proposition}
\begin{proof}
Note that $G$ is  a solution of (\ref{weak w.g.i.}) if and only if $A_1BG=(BG)^*A_1$, $ A_1BGP_{R(B)}=A_1P_{R(B)}, R(I-GB)\subseteq N(B), A_2(I-GB)=(I-GB)^*A_2$ and $A_2(I-GB)P_{N(B)}=A_2P_{N(B)}$. Equivalently, $A_1BG=(BG)^*A_1, A_1BGB=A_1B, B(I-GB)=0, A_2GB=(GB)^*A_2$, or $BGB=B, A_1BG=(BG)^*A_1$ and $A_2GB=(GB)^*A_2$.
\end{proof}

\smallskip

\begin{corollary}\label{equivalencias weak w.g.i.}
The following statements are equivalent:
\begin{enumerate}
 \item system   (\ref{weak w.g.i.}) admits a solution,
 \item the pairs $(A_1, R(B))$ and $(A_2, N(B))$ are compatible,
 \item $B$ admits a weighted generalized inverse.
\end{enumerate}
\end{corollary}

\begin{proof}
$\textit{1}\rightarrow \textit{2}$: It is straightforward. $\textit{2}\rightarrow \textit{3}$: It follows by \cite[Theorem 3.1]{[CorMae05]}. $\textit{3}\rightarrow \textit{1}$:
The assertion follows by the previous proposition.
\end{proof}

\smallskip
By Propositon \ref{caracterizacion [2]A1A2-PI} and Corollary \ref{equivalencias weak w.g.i.} it follows that  (\ref{weighted generalized inverses}) has a solution if and only if (\ref{weak w.g.i. 3 ecuaciones}) has a solution.

\medskip




\begin{thebibliography}{HD}


\normalsize
\baselineskip=17pt


\bibitem{[AndTra75]}
{ W.~Anderson and G.~Trapp}, {\em {Shorted operators. II}}, SIAM J. Appl.
  Math. 28 (1975), ~60--71.


\bibitem{[AntCorRuiSto05]}
{ J.~Antezana, G.~Corach, M.~Ruiz, and D.~Stojanoff}, {\em {Nullspaces and
  frames}}, J. Math. Anal. Appl. 309 (2005), ~709--723.

\bibitem{[Att65]}
{ M.~Atteia}, {\em {Generalization de la d\'efinition et des propri\'et\'es
  des "spline-fonctions"}}, C.R. Acad. Sci. Paris 260 (1965), ~3550--3553.

\bibitem{[Car78]}
{ S.~Caradus}, {\em {Generalized inverses and operator theory}}, Tech.
  Report, {Queen's Pap. Pure Appl. Math. 50}, 1978.

\bibitem{[CasSuc08]}
{ G.~Cassier and L. Suciu.}, {\em {Mapping theorems and similarity to contractions for classes of A-contractions}}. Hot Topics in Operator Theory, Theta Series in Advanced Mathematics (2008), 39--58.


\bibitem{[CorGirMae09]}
{ G.~Corach, J.~I. Giribet, and A.~Maestripieri}, {\em {Sard's approximation
  processes and oblique projections}}, Studia Math. 194 (2009), ~65--80.

\bibitem{[CorMae05]}
{ G.~Corach and A.~Maestripieri}, {\em {Weighted generalized inverses,
  oblique projections, and least-squares problems}}, Numer. Funct. Anal.
  Optim. 26 (2005), ~659--673.

\bibitem{[CorMaeSto01]}
{ G.~Corach, A.~Maestripieri, and D.~Stojanoff}, {\em {Oblique projections
  and Schur complements}}, Acta Sci. Math. (Szeged) 67 (2001), ~337--356.

\bibitem{[CorMaeSto02]}
{ G.~Corach, A.~Maestripieri, and D.~Stojanoff}, {\em {Generalized Schur
  complements and oblique projections}}, Linear Algebra Appl. 341 (2002),
  ~259--272.

\bibitem{[CorMaeSto02_splines]}
{ G.~Corach, A.~Maestripieri, and D.~Stojanoff}, {\em {Oblique projections
  and abstract splines}}, J. Approx. Theory 117 (2002), ~189--206.

\bibitem{[CorMaeSto05]}
{G.~Corach, A.~Maestripieri, and D.~Stojanoff}, {\em {A classification of
  projectors}}, Banach Center Publ. 67 (2005), ~145--160.

\bibitem{[CorMaeSto06]}
{G.~Corach, A.~Maestripieri, and D.~Stojanoff}, {\em {Projections in
  operator ranges}}, Proc. Amer. Math. Soc. 134 (2006), ~765--778.

\bibitem{[Dou66]}
{ R.~Douglas}, {\em {On majorization, factorization, and range inclusion of
  operators on Hilbert space}}, Proc. Amer. Math. Soc. 17 (1966), ~413--415.

\bibitem{[Eld03]}
{ Y.~C. Eldar}, {\em {Sampling with arbitrary sampling and reconstruction
  spaces and oblique dual frame vectors}}, J. Fourier Anal. Appl. 9 (2003),
 ~77--96.

\bibitem{[Eldwer05]}
{ Y.~C. Eldar and T.~Werther}, {\em {General framework for consistent
  sampling in Hilbert spaces}}, Int. J. Wavelets Multiresolut. Inf. Process.
  3 (2005), ~497--509.

\bibitem{[Eld80]}
{ L.~Eld\'en}, {\em {Perturbation theory for the least squares problem with
  linear equality constraints}}, SIAM J. Numer. Anal. 17 (1980),
  ~338--350.

\bibitem{[FilWil71]}
{P.~A. Fillmore and J.~P. Williams}, {\em {On operator ranges}}, Adv.
  Math. 7 (1971), ~254--281.

\bibitem{[GirMaeMPe10]}
{ J.~I. Giribet, A.~Maestripieri, and F.~Mart{\'\i}nez~Per{\'\i}a}, {\em {A
  geometrical approach to indefinite least squares problems}}, Acta Appl.
  Math. 111 (2010), ~65--81.

\bibitem{[HasNor94]}
{S.~Hassi and K.~Nordstr\"om}, {\em {On projections in a space with an
  indefinite metric}}, Linear Algebra Appl. 208-209 (1994), ~401--417.

\bibitem{[MaMPe06]}
{ A.~Maestripieri and F.~Mart\'{\i}nez~Per\'{\i}a}, {\em {Decomposition of
  selfadjoint projections in Krein spaces}}, Acta Sci. Math. (Szeged) 72 (2006),
  ~611--638.


\bibitem{[MitRao74]}
{ S.~K. Mitra and C.~Rao}, {\em {Projections under seminorms and generalized
  Moore Penrose inverses}}, Linear Algebra Appl. 9 (1974), ~155--167.

\bibitem{[Nas81]}
{ M.~Z. Nashed}, {\em {On generalized inverses and operator ranges}},
\newblock {Functional Analysis and Approx. Birkh\"auser, Basel (1981)   85-96.} 

\bibitem{[Nas87]}
{ M.~Z. Nashed}, {\em {Inner, outer, and
  generalized inverses in Banach and Hilbert spaces}}, Numer. Funct. Anal.
  Optim. 9 (1987), ~261--325.

\bibitem{[RaoMit73]}
{ C.~Rao and S.~K. Mitra}, {\em {Theory and application of constrained
  inverse of matrices}}, SIAM J. Appl. Math. 24 (1973), ~473--488.

\bibitem{[Sar52]}
{ A.~Sard}, {\em {Approximation and variance}}, Trans. Amer. Math. Soc. 73
  (1952), ~428--446.



\bibitem{[Suc06]}
{ L.~Suciu}, {\em {Ergodic properties for regular $A$-contractions}}, Integral Equations Operator Theory 56 (2006), ~285--299. 


\bibitem{[Suc09]}
{ L.~Suciu}, {\em {Quasi-isometries in semi-Hilbertian spaces}}, Linear
  Algebra Appl. 430 (2009), ~2474--2487.



\bibitem{[TakTiaYan07]}
{Y. Takane, Y. Tian and H. Yanai}, {\em {On constrained generalized inverses of matrices and their properties}}, Ann. Inst. Statist. Math. 59 No. 4 (2009), ~807--820.


\bibitem{[TiaTak09]}
{ Y. Tian and Y. Takane}, {\em {On $ V$-orthogonal projectors associated with a semi-norm}}, Ann. Inst. Statist. Math. 61 No. 2 (2009), ~517--530.


\end{thebibliography}
\end{document}